\documentclass[12pt]{article}
\usepackage[utf8]{inputenc}
\usepackage[margin=1in]{geometry}
\usepackage{amsmath}
\usepackage{mathrsfs}
\usepackage{amssymb}
\usepackage{amsfonts}
\usepackage{mathtools}
\usepackage{amsthm}
\usepackage{tikz-cd}
\usepackage{thmtools}
\usepackage{accents}
\usepackage{thm-restate}
\usepackage{hyperref}
\usepackage{CommonCommand}
\usepackage{subfiles} 

\hypersetup{colorlinks=true,linkcolor=blue, linktocpage}

\theoremstyle{plain}
\newtheorem{thm}{Theorem}[section]
\newtheorem{lemma}[thm]{Lemma}

    \newtheoremstyle{TheoremNum}
        {\topsep}{\topsep}              
        {\itshape}                      
        {}                              
        {\bfseries}                     
        {.}                             
        { }                             
        {\thmname{#1}\thmnote{ \bfseries #3}}
    \theoremstyle{TheoremNum}
    \newtheorem{thmn}{Theorem}

\theoremstyle{definition}
\newtheorem{defn}[thm]{Definition}
\newtheorem{exmp}[thm]{Example}
\newtheorem{rmk}[thm]{Remark}

\newcommand{\C}{\bb{C}}
\newcommand{\A}{\bb{A}}

\renewcommand{\O}{\mathcal{O}}
\newcommand{\mf}[1]{\mathfrak{#1}}
\newcommand{\m}{\mf{m}}
\DeclareMathOperator{\Spec}{Spec}
\DeclareMathOperator{\len}{len}

\title{Milnor and Tjurina numbers for an isolated complete intersection singularity}
\author{A.J. Parameswaran , Mohit Upmanyu}

\begin{document}

\maketitle

\begin{abstract}
This paper aims to prove that given a isolated complete intersection singularity, the Milnor number will be bounded by a bound depending only on Tjurina number and dimension of the singularity. The proof uses A$\m$AC (introduced in \cite{2022ParamMohit}) and as with such methods, the bound is purely existential.
\end{abstract}


\section{Introduction}

Throughout this paper we fix $A = \bb{C}[\![x_1,x_2,...,x_{n+k}]\!]$ and $\m$ denote its maximal ideal. 

\begin{thm}[\cite{liu2018milnor}]
Let $X \subset \Spec(A)$ be an isolated hypersurface singularity of dimension $n$ (i.e k = 1) with the Tjurina number $\tau$. Then $\mu(X) \leq (n+1)\tau$.
\end{thm}

In \cite{DimcaGreuel}, $\mu \leq \frac{4}{3}\tau$ was proposed for the case of plane curves, which was later proven in \cite{2020minimal} and \cite{2019saitos} separately for the case of irreducible curves. In \cite{2021Almiron}, the bound for all curves was proven. Moreover, in \cite{2021Almiron}, it is proven that if the Durfee conjecture holds for a hypersurface $X$ of dimension 2 then, $\mu \leq \frac{3}{2}\tau$.

For complete intersection singularities, a bound of  $\mu \leq \frac{4}{3}\tau$ was proved for certain families of curve in \cite{2021Almiron}. We prove an existential bound for the case of complete intersection as follows.

\newcommand{\milnorTjurinafst}{Let $X \subset \Spec(A)$ be an isolated complete intersection singularity of dimension n with Tjurina number $\tau$. Then there } 
\newcommand{\milnorTjurinasnd}{exists a constant $c>0$ \st $\mu(X) \leq c$. This $c$ depends only on $\tau$ and $n$.}

\begin{thmn}[\ref{milnor tjurina}]
\milnorTjurinafst
abstractly\footnote{We use the word abstractly to indicate that what we give is an existential argument. More precisely, what we prove is that the invariant is not infinite so we conclude that a bound exists.}
\milnorTjurinasnd
\end{thmn}

Here $\mu(X)$ refers to the Milnor number, for a definition see \cite[5.11]{LooijengaICIS}.

\section{Preliminaries}

Let $R$ be a noetherian ring of dimension $d$ with a maximal ideal $\m$. Let $I$ be an $\m$-primary ideal. We will  use $\len(M)$ to denote length of a module $M$. 

\begin{defn}\label{multi definition}
Define Samuel function $\chi_{R}^{I}(t) := \len(R/I^{t+1})$.
If the maximal ideal is clear from the context and $I=\m$, then we will drop the $\m$ in the notation.
We take the definition of multiplicity as defined in \cite[Formula 14.1]{MatsumuraCRT}.
$$e(R,I) = \lim_{t\to\inf} d! \frac{\len(R/I^{t+1})}{t^d} $$
\end{defn}

We will use the following fundamental result of dimension theory.
\begin{lemma}\label{dimension}
Let $R$ be a noetherian ring with Krull dimension $d$ with maximal ideal $\m$. Let $I$ be an $\m$-primary ideal.
Then for $t$ sufficiently large, $\chi^{I}_{R}(t) = p(t)$, where $p(t)$ is a rational polynomial of degree $d$ (called Samuel polynomial). Furthermore, the coefficient of $x^d$ is $e(R,I)/d!$.
\end{lemma}
\begin{proof}For the degree of polynomial see \cite[Theorem 13.4]{MatsumuraCRT}. For the multiplicity as coefficient see \cite[Section 14]{MatsumuraCRT}\end{proof}

\begin{lemma}\label{dim bound lemma}
Let $R$ be a noetherian ring of dimension $d$ with a maximal ideal $\m$ \st $R/\m$ is infinite. Let $\chi^I_R(t)$ be the Samuel function of $R$ w.r.t. $I$ which is an $\m$-primary ideal. Then 
\begin{equation}\chi^I_R(t) \geq {t+d\choose d}\end{equation}
Let $e$ be the multiplicity of $R$ w.r.t  $\m$. Then we can obtain better bounds as follows.
\\
For $t \leq e-1$,
\begin{equation}\chi^I_R(t) \geq {t+d+1\choose d+1} \end{equation} 
For $t \geq e - 1$,
\begin{equation}\chi^I_R(t) \geq \sum_{i=0}^{d} (-1)^{i} {e \choose i+1}{t+d-i\choose d-i} = e{t+d\choose d} + \text{lower order terms (of $t$)} \end{equation}
\end{lemma}
Note that ${t+d\choose d}$ is the Samuel polynomial of a regular local ring of dimension $d$
\begin{proof}
In \cite[lemma 2.6]{2022ParamMohit}, this lemma is proven for $I = \m$, The same proof generalizes directly for any $\m$-primary ideal $I$.
\end{proof}

\begin{defn}\label{singular Definition}
Let $\A^k = \Spec \C[x_1,...,x_k]$. 
Let $f = (f_1,...,f_{k}) : Spec A \to \A^{k}$ and let $X := f^{-1}(0)$ considered as a closed subscheme of $\Spec(A)$.
Define $\S^{\a}_{X}$ as the scheme of ``all points $p \in X$ \st $\dim(T_p(X)) > \a$", whose scheme structure is given as follows.

$$\S^{\a}_{X} := X \cap \Spec\frac{A}{J_{{n+k} -\a}(f)} = \Spec\frac{A}{I(X)+J_{{n+k} -\a}(f)}$$
where $J_{{n+k}-\a}(f)$ is the ideal generated by ${n+k}-\a \cross {n+k}-\a$ minors of the Jacobian matrix $(\frac{\doe f_i}{\doe x_j})$.

As seen in \cite[4.D]{LooijengaICIS}, $J_{{n+k}-\a}(f)$ corresponds to the Fitting ideal of the module $\W_f := Coker(A^k = (f^*(\W_{\A^k})) \xra{Jac(f_i)} A^{n+k} (= \W_{\Spec A})$, so the scheme $\S^{\a}_X$ does not depend on the choice of generators of the ideal.

We define the critical locus of the map $f$ as the scheme of ``all points $p \in \Spec A$ \st $rank (Df_p(X)) < k$", whose scheme structure is given as follows.

$$C_f := \Spec\frac{A}{J_{k}(f)}$$
\end{defn}

\begin{defn}
Let $f = (f_1,...,f_{k}) : Spec A \to \A^{k}$. Let $X$ be the closed subscheme of $\Spec(A)$ defined as $f^{-1}(0)$. Then $X$ is called an \textbf{isolated compete intersection singularity} (icis) if $\dim(\S^n_X) = 0$. 

Following \cite[6.B]{LooijengaICIS},
if $X$ is an icis then we define $T^1_X$ as the cokernel of $T(\Spec A)|_X \to f^{*}T(\A^{k})|_X$. Since the support of $T^1_X$ is on singular locus which is of dimension 0. So we define the Tjurina number $\tau(X)$ as the $\len (T^1_X)$.
\end{defn}

\begin{rmk}
Note that as defined above, $(f_1,...,f_{k})$ define a complete intersection iff $(f_1,...,f_{k})$ form a regular sequence. This we will take as a convention.
\end{rmk}

Here we define algebraic $\m$-adically closed subsets introduced in \cite{2022ParamMohit}, on which the theory is built.

\begin{defn}
A subset $P \subset A^{k}$ (Cartesian  product) is said to be an \textbf{algebraic $\m$-adically closed (A$\m$AC) class}, if $P = \ula{\lim} P_i$ where $P_i$ are constructible subsets of $(A/\m^i)^{k}$ (in Zariski topology).

If $U$ is a set complement of an A$\m$AC then we refer to it as an \textbf{algebraic $\m$-adically open (A$\m$AO) class}.

Finally, when we say "Let $P =  \ula{\lim}P_i$ be an A$\m$AC class in $A^{k}$" , It will be assumed that $P_i \subset (A/\m^i)^{k}$ and $P \subset A^{k}$
\end{defn}

We restate some  properties of A$\m$AC proved in \cite{2022ParamMohit}

\begin{lemma}\cite{2022ParamMohit}[lemma 3.6]\label{basic lemma}
Let $P^1 = \ula{\lim}P^1_i$, $P^2 = \ula{\lim}P^2_i$,..., $P^j = \ula{\lim}P^j_i$ ,...  be an infinite sequence of A$\m$AC classes in $A^{k}$. 
\begin{enumerate}
\item A$\m$AC classes are closed under finite union. ($P^1 \cup P^2 =  \ula{\lim} P^1_i \cup P^2_i$)
\item A$\m$AC classes are closed under finite intersections. ($P^1 \cap P^2 =  \ula{\lim} P^1_i \cap P^2_i$)
\item A$\m$AC classes are closed under countable intersections. ($\bigcap_j P^j =  \ula{\lim} \bigcap_{j\leq i} P^j_i$ )
\end{enumerate}
\end{lemma}

\begin{lemma}\label{countable intersection lemma}
Let $P^1 = \ula{\lim}P^1_i$, $P^2 = \ula{\lim}P^2_i$,..., $P^j = \ula{\lim}P^j_i$ ,...  be an infinite sequence of A$\m$AC classes in $A^{k}$, Then countable intersection of $P_i$ is non-empty iff each finite intersection of $P_i$ is non-empty.
\end{lemma}
\begin{proof}
By previous lemma $\bigcap_j P^j =  \ula{\lim} \bigcap_{j\leq i} P^j_i$ 

Now by non emptyness lemma \cite{2022ParamMohit}[lemma 3.4]

\[\bigcap_j P^j \neq \emptyset \iff \ula{\lim} \bigcap_{j\leq i} P^j_i \neq \emptyset  \forall i \iff \bigcap_{j\leq i} P^j  \neq \emptyset \forall i\]
\end{proof}

\section{Main Theorem}

Before we get to the Theorem we would need to prove the following lemma.

\begin{lemma}\label{inequality lemma}
Let $f = (f_1,...,f_k) : \Spec A \to \A^k$. Let $X := f^{-1}(0) \subset \Spec(A)$ be an isolated complete intersection singularity. Then $\mu(X) \leq$ the multiplicity of critical locus w.r.t. $\<f_1,...,f_k\>$
\end{lemma}
\begin{proof}
We use formula from \cite[5.11.a]{LooijengaICIS}, to conclude that 
\[\mu(X) \leq \len(\frac{A}{J_{k}(f)+\<g_1,...,g_{k-1}\>})\]
Here $g_1,...,g_{k-1}$  can be any linear combinations of $f_1,...,f_k$.

Now By \cite[4.4]{LooijengaICIS} $\O_{C_f}$ is finite and Cohen-Macaulay over $\C[\![f_1,...,f_{k}]\!]$ of dimension $k-1$. So for $k-1$ general linear combinations $g_1,...,g_{k-1}$, $\O_{C_f}$ is finite and Cohen-Macaulay over $\C[\![g_1,...,g_{k-1}]\!]$. Now the projective dimension of $\O_{C_f}$ over $\C[\![g_1,...,g_{k-1}]\!]$ is $\dim \C[\![g_1,...,g_{k-1}]\!] - depth(\O_{C_f})$ = $k-1 - \dim \O_{C_f}$ = $0$. So $\O_{C_f}$ is a free module over $\C[\![g_1,...,g_{k-1}]\!]$

We can also assume $e(\O_{C_f} , \<f_1,...,f_{k}\>)$ = $e(\O_{C_f} , \<g_1,...,g_{k-1}\>)$, using \cite[Theorem 14.14]{MatsumuraCRT}. 

Since $\O_{C_f}$ is a free module over $\C[\![g_1,...,g_{k-1}]\!]$. 
\[e(\O_{C_f} , \<g_1,...,g_{k-1}\>) = rank (\O_{C_f}) = \len \frac{\O_{C_f}}{\<g_1,...,g_{k-1}\>}  = \len(\frac{A}{J_{k}(f)+\<g_1,...,g_{k-1}\>})\].

We can finally combine all the above to get
\[\mu(X) \leq \len(\frac{A}{J_{k}(f)+\<g_1,...,g_{k-1}\>}) = e(\O_{C_f} , \<f_1,...,f_{k}\>)\]
\end{proof}

Using the observation that $\mu(X) \leq$ the multiplicity of critical locus w.r.t. $\<f_1,...,f_\tau\>$, bounding $\mu$ reduces to bounding $e(\O_{C_f} , \<f_1,...,f_{k}\>)$. For bounding $e(\O_{C_f} , \<f_1,...,f_{k}\>)$  we require the following A$\m$AC classes. (The proof of the fact that they are A$\m$AC is given in section 4)

\begin{exmp}\label{some example}
\begin{enumerate}
\item $T(r) = \{(f_1,...,f_{k}) \in A^{k} \ |\ (f_1,...,f_{k}) \text{ define an icis and } \tau(f_1,...,f_{k}) = r\}$ (for proof see \ref{Tjurina AmAC})

\item $D(d) := \{(f_1,...,f_{k}) \in A^{k} \ |\ \dim(\S^{n}_{X}) \geq d\}$ (For definition of $\S^{n}_{X}$ see \ref{singular Definition}.) (for proof see \ref{dimension AmAC})

\begin{rmk}
One notes that $(f_1,...,f_{k}) \in D(1)$ iff $(f_1,...,f_{k})$ do \textbf{not} define an icis.
The fact that we have a characterization of not icis as an A$\m$AC will prove very useful.
\end{rmk}

\item 
\begin{align} \nonumber
\begin{split}
C(e) := \{(f_1,...,f_{k}) \in A^{k} \ |&\ \text{either } \dim(\S^{n}_{X}) \geq 1 \text{ or } (f_1,...,f_{k})  \text{ define an icis and } \\ & e(\O_{C_{f}}, \<f_1,...,f_{k}\>) \geq e\} \text{ (for proof see \ref{critical AmAC})}
\end{split}
\end{align}

\end{enumerate}
\end{exmp}

\begin{thm}\label{tjurina critical}
Let $f = (f_1,...,f_k) : \Spec A \to \A^k$. Let $X := f^{-1}(0) \subset \Spec(A)$ be an isolated complete intersection singularity with Tjurina number $\tau$. Then there abstractly exists a constant $c>0$ \st the multiplicity of $C_{f}$ w.r.t. $\<f_1,...,f_k\>$ $\leq c$. This $c$ depends only on $\tau$ and $n$
\end{thm}
\begin{proof}
Note That $D(1) = \bigcap_e C(e)$. 

For this proof we fix $k = \tau$. Since Tjurina number is the dimension of base of a mini-versal deformation \cite{LooijengaICIS}, we see that all possible icis with Tjurina number $\tau$ must occur within codimension $k = \tau$. As a consequence $T(\tau)$ is a collection of all possible icis which satisfy the hypothesis of the Theorem (albeit with repetition up to isomorphism). (Using this we have shown the bound does not depend on $k$)

Assume for a contradiction that the multiplicity of $C_{f}$ w.r.t. $\<f_1,...,f_\tau\>$ is not bounded on the set $T(\tau)$.
\begin{align} \nonumber
\begin{split}
& T(\tau) \cap C(e) \neq \emptyset \quad \forall e \text{ sufficiently larger than } \tau\\
\implies &T(\tau) \cap \bigcap_e C(e) \neq \emptyset \quad \text{ (by Lemma \ref{countable intersection lemma})}\\
\implies &T(\tau) \cap D(1) \neq \emptyset
\end{split}
\end{align}
This contradicts the fact that any $(f_1,...,f_{\tau}) \in T(\tau)$ can only define icis.
\end{proof}

\begin{thm}\label{milnor tjurina}
\milnorTjurinafst
abstractly
\milnorTjurinasnd
\end{thm}
\begin{proof}
We use \ref{inequality lemma} to conclude $\mu \leq$ the multiplicity of critical locus w.r.t. $\<f_1,...,f_\tau\>$, so the result follows from \ref{tjurina critical} 
\end{proof}

\section{Proof of Examples}

Here we give proves of the examples stated and used in Theorem \ref{tjurina critical}

\begin{lemma}\label{Tjurina AmAC}
$T(r)$ (defined in \ref{some example}.1) is both A$\m$AC and A$\m$AO
\end{lemma}
\begin{proof}
We look at $B_{r+2} = (A/\m^{r+2})^{k}$ regarded as an affine space. Define $\pi: A^{k} \to B_{r+2}$ to be the canonical quotient map.

Using semi continuity of Tjurina number, $\-T := \{(\-f_1,...,\-f_{k})\in  B_{r+2} \ | \ \len(T^1_X/\m^{r+1}T^1_X) = r\}$ is a Zariski constructible subset.

We will prove that $T(r) = \pi^{-1}(\-T)$ 

First, let $(f_1,...,f_{k}) \in T(r)$. Since $\len(T^1_X) = r$, $\m^{r}T^1_X = 0$, so  $T^1_X = T^1_X/\m^{r+1}T^1_X$, so $\pi(f_1,...,f_{k})\in\-T$

Let $(f_1,...,f_{k}) \in \pi^{-1}(\-T)$. 
\begin{align} \nonumber
\begin{split}
&\len(\frac{T^1_X}{\m^{r+1}T^1_X}) = r\\ 
\implies & \m^{r}(\frac{T^1_X}{\m^{r+1}T^1_X}) = 0 \\
\implies & \frac{\m^{r}T^1_X}{\m^{r+1}T^1_X} = 0 \\
\implies & \m^{r}T^1_X = \m^{r+1}T^1_X \\
\implies & \m^{r}T^1_X = 0 \text{ (By Nakyama lemma)} \\
\implies & T^1_X = T^1_X/\m^{r+1}T^1_X\\
\end{split}
\end{align}
So we can conclude $\len(T^1_X) = r$, hence $(f_1,...,f_{k}) \in T(r)$.
\end{proof}

\begin{lemma}\label{dimension AmAC}
$D(d)$ (defined in \ref{some example}.3) is an A$\m$AC.
\end{lemma}
\begin{proof}
We will prove that $D(d) = \ula{\lim} D_i(d)$, where $D_i(d)$ is defined as follows.

$$D_r(d) := \{ (\-f_1,...,\-f_{k}) \in (A/\m^{r})^{k} \ |\ \len(\O_{\S^{n}_{X}}/\m^{t+1}) \geq {t+d \choose d}, \ \forall \ t \leq r-2\}$$
One sees that, $D_r(d)$ is a Zariski constructible subset. (One argument can be seen in \cite{2022ParamMohit}[lemma 2.9].)

$D(d) \subset \ula{\lim} D_i(d)$ because if $\dim(\S^{n}_{X}) \geq r$ then $(\-f_1,...,\-f_{k}) = (f_1 + \m^r ,..., f_{k} + \m^r)$ satisfies the condition of $D_r(d)$ by lemma \ref{dim bound lemma}(1).

Let $(f_1,...,f_{k}) \in A$ and $(f_1,...,f_{k}) \not\in D(d)$, then $\len(\O_{\S^{n}_{X}}/\m^{t+1})$ is a polynomial of degree $< d$ (lemma \ref{dimension}). So for large $t$, $\len(\O_{\S^{n}_{X}}/\m^{t+1})<{t+d \choose d}$ which is a higher degree polynomial. So $(f_1,...,f_{k}) \not\in \ula{\lim} D_i(d)$
\end{proof}

\begin{lemma}\label{critical AmAC}
$C(e)$ (defined in \ref{some example}.2) is an A$\m$AC
\end{lemma}

\begin{proof}
\begin{align} \nonumber
\begin{split}
C_r := \{(\-f_1,...,\-f_{k}) \in  (A/\m^{r})^{k} \ &|\ \len(\frac{\O_{C_{F}}}{\m^{e{t + k - 1 \choose k - 1}} + \<f_1,...,f_{k}\>^{t+1}}) \geq \sum_{i=0}^{d} (-1)^{i} {e \choose i+1}{t +k-1-i\choose k-1-i},\\ &\forall \ t \text{ \st } e-1 \leq t \text{ and } e{t + k - 1 \choose k - 1} \leq r-2\}
\end{split}
\end{align}
One sees that, $C_r$ is a Zariski constructible subset. (Similar to the previous Lemma)

We will prove that $C(e) = \ula{\lim} C_i \cup D(1)$.

$C(e) \subset \ula{\lim} C_i \cup D(1)$ because if $(f_1,...,f_{k}) \in C(e)$ and $(f_1,...,f_{k}) \not\in D(1)$ then  $(f_1,...,f_{k})$ is an icis, and so $(\-f_1,...,\-f_{k}) = (f_1,...,f_{k})+\m^{r}$ satisfies the condition of $C_i$ by lemma \ref{dim bound lemma}(3).

First Note that from definition of $C(e)$ is $D(1) \subset C(e)$

Let $(f_1,...,f_{k}) \in A^{k}$, and $(f_1,...,f_{k}) \in \ula{\lim} C(e) - D(1)$ then $(f_1,...,f_{k})$ defines an icis. Then $\len(\frac{C_{f}}{\<f_1,...,f_{k}\>^{t+1}})$ makes sense and is a polynomial of degree $> d$ (in this case it is not an icis and is in $D(1)$) or if degree $= d$ the highest coefficient $\geq e/d!$ (lemma \ref{dimension} and definition of multiplicity). So Critical locus $C_{f}$ has multiplicity $\geq e$. So $(f_1,...,f_{k}) \in C(e)$
\end{proof}

\renewcommand{\abstractname}{Acknowledgements}
\begin{abstract}
This work was supported by the Department of Atomic Energy, Government of India [project no. 12 - R\&D - TFR - 5.01 - 0500].
\end{abstract}
\renewcommand{\abstractname}{Abstract}

\bibliographystyle{alpha}
\bibliography{references}

\end{document}